\documentclass[a4paper,10pt]{article}

\usepackage{enumerate} %allows for roman and other numerals
\usepackage{amsmath}  %
\usepackage{amsfonts} % Standard AMS packages
\usepackage{amssymb}  %
\usepackage{amsthm}   %
\newcommand{\VV}{\mathbb{V}}

\DeclareMathOperator{\Ann}{Ann}

\newcommand{\RR}[0]{\mathbb{R}}  % reals
\newcommand{\CC}[0]{\mathbb{C}}  % complex
\newcommand{\NN}[0]{\mathbb{N}}  % natural numbers
\newcommand{\WW}[0]{\mathbb{W}}  % whole numbers
\newcommand{\ZZ}[0]{\mathbb{Z}}  % integers
\newcommand{\DD}{\mathbb{D}}     % open unit disc in \CC
\newcommand{\TT}{\mathbb{T}}     % the "1-torus"
\newcommand{\inner}[2]{ \langle #1 , #2 \rangle} %inner prod.
\DeclareMathOperator{\ran}{ran}
\newcommand{\pd}[0]{\partial} %simpler notation for partial derivatives
\newcommand{\sbs}{\subset}

\newcommand{\sbse}{\subseteq}

\newcommand{\frk}[1]{\mathfrak{#1}} % Fraktur text
\newcommand{\cc}[1]{\overline{#1}} %Complex conjugate
\newcommand{\of}{\circ} % E.g. $f \of g$ read "f of g"
\DeclareMathOperator{\Hom}{Hom}   % The homomorphism/morphism group of something
\newcommand{\mc}[1]{\mathcal{#1}} % Math style calligraphy
\newcommand{\bksl}{\backslash}
\newcommand{\gra}{\alpha}
\newcommand{\grb}{\beta}

\newcommand{\grd}{\delta}

\newcommand{\grg}{\gamma}

\newcommand{\grk}{\kappa}
\newcommand{\grl}{\lambda}

\newcommand{\grs}{\sigma}

\newcommand{\grw}{\omega}

\newcommand{\grz}{\zeta}

\newcommand{\grD}{\Delta}

\newcommand{\grG}{\Gamma}
\newcommand{\grJ}{\Theta}

\newcommand{\grW}{\Omega}

\theoremstyle{plain}
\newtheorem{theorem}{Theorem}[section]
\newtheorem{lemma}[theorem]{Lemma}
\newtheorem{proposition}[theorem]{Proposition}
\newtheorem{corollary}[theorem]{Corollary}

\newtheoremstyle{named}{}{}{\itshape}{}{\bfseries}{.}{.5em}{#2. \thmnote{#3}}
\theoremstyle{named}

\theoremstyle{definition}

\newtheorem{example}[theorem]{Example}
\theoremstyle{remark}

\title{A classification of $n$-tuples of commuting shifts of finite multiplicity}
\author{Edward J. Timko}
\date{}

\begin{document}

\maketitle
\begin{abstract}
  Let $\VV$ denote an $n$-tuple of shifts of finite multiplicity, and denote by $\Ann(\VV)$ the ideal consisting of polynomials $p$ in $n$ complex variables such that $p(\VV)=0$. If $\WW$ on $\frk{K}$ is another $n$-tuple of shifts of finite multiplicity, and there is a $\WW$-invariant subspace $\frk{K}'$ of finite codimension in $\frk{K}$ so that $\WW|\frk{K}'$ is similar to $\VV$, then we write $\VV\lesssim \WW$. If $\WW\lesssim \VV$ as well, then we write $\WW\approx \VV$.
  
  In the case that $\Ann(\VV)$ is a prime ideal we show that the equivalence class of $\VV$ is determined by $\Ann(\VV)$ and a positive integer $k$. More generally, the equivalence class of $\VV$ is determined by $\Ann(\VV)$ and an $m$-tuple of positive integers, where $m$ is the number of irreducible components of the zero set of $\Ann(\VV)$.
\end{abstract}

\section{Introduction}

An isometry $V$ on a (complex) Hilbert space $\mathfrak{H}$ is called a \emph{shift} of \emph{multiplicity} $k$ when $V$ is unitarily equivalent to the standard unilateral shift on $\ell^2 (\NN) \otimes \CC^k$ for some positive integer $k$. By the von Neumann-Wold theorem, $V$ is a shift of multiplicity $k$ if and only if $\bigcap_{j = 0}^{\infty} V^j\mathfrak{H}= \{ 0 \}$ and $\dim \ker V^{\ast} = k$. For the entirety of this paper, we fix an integer $n\geq 2$. Given an $n$-tuple $\VV= ( V_1,\ldots, V_n)$ of commuting shifts of finite multiplicity, the \emph{annihilator} of $\VV$ is defined to be the polynomial ideal
\[ \Ann ( \VV) = \{ p \in \CC [ x_1, \ldots, x_n] : p ( \VV) = 0 \} , \]
It is known from \cite[Prop. 6.3]{Timko} that $\Ann ( \VV)$ is a non-trivial ideal and that it determines a variety of pure dimension 1. We briefly review these facts in Section 2.

Suppose that $\VV$ and $\WW$ are $n$-tuples of commuting shifts of finite multiplicity on Hilbert spaces $\mathfrak{H}$ and $\mathfrak{K}$, respectively. If there exists a $\WW$-invariant subspace $\mathfrak{K}'$ of finite codimension in $\frk{K}$ such that $\WW|\mathfrak{K}'$ is similar to $\VV$, then we write $\VV \lesssim \WW$. We say that $\VV$ and $\WW$ are \emph{virtually similar} and write $\VV\approx \WW$ when $\WW \lesssim \VV$ and $\VV \lesssim \WW$. 

The finite multiplicity of the elements of $\VV$ implies that $\VV$ has a finite cyclic set. In particular, there exists a set $\{h_1, \ldots, h_k\}\sbs \mathfrak{H}$ of least cardinality such that the subspace $\bigvee_{i = 1}^k \{ p (\VV) h_i : p \in \CC [ x_1, \ldots, x_n] \}$ has finite codimension in $\mathfrak{H}$. We call $k$ the \emph{virtually cyclicity} of $\VV$, and denote it by $\grk(\VV)$.

For the case in which $\Ann ( \VV)$ is a prime ideal, we show in Theorem \ref{MainThm} that $\VV$ and $\WW$ are virtually similar if and only if $\Ann ( \VV) = \Ann ( \WW)$ and $\kappa ( \VV) = \kappa ( \WW)$. When $\Ann ( \VV)$ is not prime, we have the following characterization. The ideal $\Ann (\VV)$ is radical and is the intersection of a unique finite set of prime ideals $\mathcal{I}_1, \ldots, \mathcal{I}_m$. For $i=1,\dots,m$, there exists a largest $\VV$-invariant subspace $\frk{H}_i^+$ such that $\Ann(\VV|\frk{H}^+_i)=\mc{I}_i$; the subspaces $\mathfrak{K}_1^+, \mathfrak{K}_2^+, \ldots$ are analogously defined. We show in Theorem \ref{GenThm} that $\VV \approx \WW$ if and only if $\Ann (\VV) = \Ann (\WW)$ and $\kappa(\VV|\mathfrak{H}_i^+) = \kappa ( \WW|\mathfrak{K}_i^+)$ for $i = 1, \ldots, m$.

We remark that virtual similarity is stronger than necessary, and that `virtual quasi-similarity' suffices. That is, it is sufficient that there exist injective operators $X:\frk{H}\to\frk{K}$ and $Y:\frk{K}\to\frk{H}$ intertwining $\VV$ with $\WW$ such that $(\ran X)^\bot$ and $(\ran Y)^\bot$ are finite dimensional.

The remainder of this paper is organized as follows. In section 2 we set notation and establish some preliminary results. In particular, we show that every $n$-tuple of commuting shifts of finite multiplicity has a non-trivial annihilator. In section 3 we extend some results from \cite{AnD} to provide a characterization of the $H^{\infty} ( R)$-invariant subspaces of the vector valued Hardy space $H^2(R,\frk{X})$, where $R$ is a sufficiently `nice' sub-domain of a compact Riemann surface. The main results discussed above appear in section 4. In section 5 we comment on another notion of equivalence for $n$-tuples that we call \emph{virtual unitary equivalence}, where the similarities appearing in the definition of virtual similarity are replaced by unitary maps. It is an open question, first essentially raised in \cite{AKM}, whether virtually similar $n$-tuples are also virtually unitarily equivalent. We show that this is true in a special case.

We would like to thank Hari Bercovici for his comments and support during the preparation of this document. We would also like to thank Greg Knese, Norm Levenberg, Noah Snyder, and Alberto Torchinsky for their conversations with the author. 

%%%%%%%%%%%%%%%%%%%%%%%%%%%%%%%%%%%%%%%%%%%%%%%%%%%%%%%%%%%%%%%%%%%%%%%%%%%%%%
\section{Preliminaries}

We start this section by reviewing why $n$-tuples of commuting shifts of finite multiplicity are always `algebraic', the proof of which is included for the readers convenience. Before this, we require the following result based on work in \cite{AgMcActa}. Here and afterward, $\DD$ denotes the open unit disc in $\CC$. 

\begin{proposition}\label{AKMGetPoly}
  Let $V_1, V_2$ be a pair of commuting shifts of finite multiplicity. There exist relatively prime polynomials $p ( x_1, x_2)$ and $q ( x_2)$ with the following properties.
  \begin{enumerate}
    \item[\rm{(i)}] $p(V_1,V_2)=0$ and the zero set of $p$ is contained in $\DD^2\cup(\pd\DD)^2\cup(\CC\bksl\overline{\DD})^2$.
    \item[\rm{(ii)}] $q$ has no zeros in the closed disc $\overline{\DD}$.
    \item[\rm{(iii)}] The polynomial $p$ is of the form
    \[ p(x_1,x_2)=q(x_2)x_1^d+a_1(x_2)x_1^{d-1}+\cdots+a_d(x_2) \]
    for some single-variable polynomials $a_1,\dots,a_d$.
  \end{enumerate}
\end{proposition}
The core of the proof can be found in \cite[Thm.1.12]{AgMcActa} and \cite[Prop. 6.3]{Timko}, but we provide details here for completeness.

\begin{proof}[Proof of Prop. \ref{AKMGetPoly}] The pair $(V_1,V_2)$ is unitarily equivalent to the pair of Toeplitz operators $(T_{\grJ},T_{\grz\: I})$ acting on the vector-valued Hardy space $H^2(\DD,\CC^k)$, where $k=\dim\ker V_2^*$, $\grz$ is the coordinate function on $\DD$, and $\grJ$ is a matrix-valued inner function. Because $\dim\ker V_1^*<\infty$, it follows from \cite[Thm. VI.3.1]{SzNagyFoias} that $\grJ$ has rational entries. There are clearly polynomials $P(x_1,x_2)$ and $Q(x_2)$ such that $\det(w\cdot I-\grJ(z))=P(w,z)/Q(z)$ and $P(V_1,V_2)=0$. It follows from \cite[Thm. 1.20]{AKM} that there is a non-zero polynomial $p$ satisfying property (i) that divides $P$. Let $q$ be such that property (iii) holds, and note that $q$ divides $Q$. It is clear that $Q$ has no roots on the closed unit disc, and therefore property (ii) holds.
\end{proof}

Given $p_1, \ldots, p_r \in \mathbb{C} [ x_1, \ldots, x_n]$, denote by $Z (p_1, \ldots, p_r)$ the set of all $z \in \mathbb{C}^n$ such that $p_1 ( z) = \ldots = p_r ( z) = 0$. More generally, if $S \subseteq \mathbb{C} [ x_1,\ldots, x_n]$, then we set $Z ( S) = \left\{ z \in \mathbb{C}^n : p ( z) = 0\text{ for all } p \in S \right\}$. Subsets of the form $Z ( S)$ are called \emph{algebraic varieties}, and we refer the reader to \cite{Kendig} for more information on this topic. We recall in particular that if $\mathcal{V}$ is an irreducible algebraic variety and $p$ is a non-trivial polynomial such that $\mathcal{V}\cap Z(p) \neq\emptyset$, then
\begin{equation}
  \dim ( \mathcal{V} \cap Z ( p)) \geq \dim ( \mathcal{V}) - 1.
  \label{DimIneqAlgGeo}
\end{equation}

\begin{corollary}
  Let $\mathbb{V}$ be an $n$-tuple of commuting shifts of finite multiplicity. Then $\Ann ( \mathbb{V})$ is
  non-trivial and $\dim Z ( \Ann ( \mathbb{V})) = 1$.
\end{corollary}

\begin{proof}
  We apply Proposition \ref{AKMGetPoly} to $( V_1, V_n), \ldots, ( V_{n - 1}, V_n)$ to produce polynomials $p_1 ( x_1, x_n), \ldots, p_{n - 1} ( x_{n -  1}, x_n)$, respectively, in the annihilator $\Ann ( \mathbb{V})$. From repeated use of (\ref{DimIneqAlgGeo}) we have that $\dim Z ( p_1, \ldots, p_n) = 1$ and therefore $d = \dim Z ( \Ann (  \mathbb{V})) \leq 1$. If it were the case that $d = 0$, then there would be a non-zero single-variable polynomial $q(x_1)\in\Ann( \mathbb{V})$. But operators of the form $q ( V_1)$ are injective and so $d = 1$.
\end{proof}

We collect additional information about $\Ann(\VV)$ in the following proposition. 
\begin{proposition}
  Let $\VV$ be an $n$-tuple of commuting shifts of finite multiplicity. The following assertions hold.
  \begin{enumerate}
    \item[\rm{(i)}] $\Ann(\VV)$ is a radical ideal.
    \item[\rm{(ii)}] Each irreducible component of $Z(\Ann(\VV))$ has dimension $1$.
    \item[\rm{(iii)}] $Z(\Ann(\VV))\sbse \DD^n\cup\TT^n\cup(\CC\bksl\cc{\DD})^n$
  \end{enumerate}
\end{proposition}
\begin{proof}
  Part (i) follows from the fact that $\VV$ is a $n$-tuple of commuting subnormal operators. In particular, if $\widetilde{\VV}$ denotes the minimal unitary extension of $\VV$, then we have the identity $\|p(\VV)\|=\|p(\widetilde{\VV})\|$ for each $p\in\CC[x_1,\dots,x_n]$. Assertion (ii) follows from \cite[Lemma 3.1]{Timko}, which demonstrates that no irreducible component of $Z(\Ann(\VV))$ has dimension 0. To prove (iii), we apply Proposition \ref{AKMGetPoly}(i) repeatedly.
\end{proof}

A \emph{finite Riemann surface} is a subdomain $R$ of a compact Riemann surface with the property that $\pd R$ is locally a real analytic curve. In particular, $\pd R$ is a finite disjoint union of topological circles. If the annihilator is a prime ideal, we can desingularize $Z ( \Ann ( \mathbb{V})) \cap \mathbb{D}^n$ to a finite Riemann surface, the pertinent properties of which are summarized in the following proposition. For proof, we refer the reader to \cite[Sec. 3]{AKM} for the case of $n=2$ and \cite[Sec. 7.1]{Timko} for $n>2$.

\begin{proposition}\label{GetR} Assume that $\Ann ( \mathbb{V})$ is a prime ideal. There exists a finite Riemann surface $R$ and a continuous proper map $\xi$ from $\overline{R}$ onto $Z ( \Ann ( \mathbb{V})) \cap \overline{\mathbb{D}}^n$ such that
  \begin{enumerate}
    \item[\rm{(i)}] $\xi ( \partial R) = Z ( \Ann ( \mathbb{V})) \cap (
    \partial \mathbb{D})^n$;
    
    \item[\rm{(ii)}] $\xi |R$ is holomorphic onto $Z ( \Ann ( \mathbb{V}))
    \cap \mathbb{D}^n$; and
    
    \item[\rm{(iii)}] there is a cofinite subset $X$ of $\cc{R}$ such that $\xi|(X\cap R)$ is a biholomorphism onto its image and $\xi|(X\cap\pd R)$ is a diffeomorphism onto its image.
  \end{enumerate}
\end{proposition}

We note that $\xi = ( \xi_1, \ldots, \xi_n)$ is an $n$-tuple of analytic functions on $R$ with unimodular boundary values. The associated multiplication operators on $H^2(R)$ are thus isometries. Of course, we also have that $p\of\xi\equiv 0$ for each $p\in\Ann(\VV)$.

Denote by $A ( R)$ the algebra of continuous function on $\overline{R}$ that are analytic on $R$. We equip $A ( R)$ with the topology of uniform convergence on $\overline{R}$ and let $A_{\xi} ( R)$ denote the closed unital subalgebra of $A ( R)$ generated by $\xi_1, \ldots, \xi_n$. We identify the elements of $A ( R)$ with their boundary value functions on $\partial R$ and thus view $A(R)$ as a subalgebra of $C(\partial R)$.

The following result appears in \cite{AKM} for the case of $n=2$. We give here a different argument for $n\geq 2$. In what follows, we set $\NN_0=\NN\cup \{0\}$.

\begin{lemma}\label{GetQLem}
  There exists a non-zero single-variable polynomial $Q$ such that $Q ( \xi_n) A ( R) \subseteq A_{\xi} ( R)$.
\end{lemma}

\begin{proof}
  If $A_{\xi} ( R)$ has finite codimension in $A ( R)$, then the lemma follows from \cite[Thm. 9.8]{GamEmb}. Let $C_{\xi}(\partial R)$ denote the closed unital $\ast$-subalgebra of $C (\partial R)$ generated by $\xi_1, \ldots, \xi_n$. The map $\xi$ separates all but a finite set of points of $\partial R$, and thus $\mathrm{Re}\: C_{\xi}(\partial R)$ has finite (real) codimension in $\mathrm{Re}\: C ( \partial R)$. From the commuting square of inclusion maps
  \[ \begin{array}{ccc}
       \mathrm{Re}\: A_{\xi} ( R) & \rightarrow & \mathrm{Re}\: C_{\xi}(\partial R)\\
       \downarrow &  & \downarrow\\
       \mathrm{Re}\: A ( R) & \rightarrow & \mathrm{Re}\: C ( \partial R)
     \end{array}, \]
  we deduce that $A_{\xi}(R)$ has finite codimension in $A(R)$ if $\mathrm{Re}\: A_{\xi} ( R)$ has finite real codimension in $\mathrm{Re}\: C_{\xi}(\partial R)$.
  
  Let $j\in\{1,\dots,n-1\}$, and note that Corollary \ref{AKMGetPoly} provides single-variable polynomials $a_0^{(j)},\dots,a_{d_j}^{(j)}$ such that
  \[ a_{d_j}^{(j)}(\xi_n)\xi_j^{d_j}+\cdots+a_0^{(j)}(\xi_n)=0 \]
  and $a_{d_j}^{(j)}$ has no zeros in $\overline{\DD}$. With $S=\prod_{i=1}^{n-1}\{j\in\ZZ:|j|<d_i\}$ and $S_+=S\cap\NN_0^{n-1}$, we find that
  \[ A_\xi(R)= \bigvee\{\xi_1^{j_1}\cdots\xi_{n-1}^{j_{n-1}}\xi_n^\ell : (j_1,\dots,j_{n-1})\in S_+,\ell\in\NN_0\} \]
  and
  \[ C_\xi(\partial R ) = \bigvee\{\xi_1^{j_1}\cdots\xi_{n-1}^{j_{n-1}}\xi_n^\ell: (j_1,\dots,j_{n-1})\in S,\ell\in\ZZ\}. \]
  In particular, $\mathrm{Re}\: C_\xi(\partial R)$ is the closed linear span of elements of the form $\mathrm{Re}\: (\xi_1^{j_1}\cdots \xi_{n-1}^{j_{n-1}}\xi_n^\ell)$ and $\mathrm{Im}\: (\xi_1^{j_1}\cdots \xi_{n-1}^{j_{n-1}}\xi_n^\ell)$ with $(j_1,\dots,j_{n-1})\in S$ and $\ell\in\NN_0$. We note that
  \[ 1/\xi_j=-a_0^{(j)}(\xi_n)^{-1}(a_{d_j}^{(j)}(\xi_n)\xi_j^{d_j-1}+\cdots+a_1^{(j)}(\xi_n)). \]
  Set $r=(a_0^{(1)})^{d_1-1}\cdots (a_0^{(n-1)})^{d_{n-1}-1}$ and denote by $\delta$ the degree of $r$. For each $(j_1,\dots,j_{n-1})\in S$ and $\ell\in\NN_0$, there exists a polynomial $b_{j_1,\dots,j_{n-1},\ell}$, of degree at most $\delta-1$ in the $n$-th variable and at most $d_j-1$ in the $j$-th variable such that
  \[ \xi_1^{j_1}\cdots \xi_{n-1}^{j_{n-1}}\xi_n^\ell-\frac{b_{j_1,\cdots,j_{n-1},\ell}(\xi)}{r(\xi_n)}\in A_\xi(R). \]
  We conclude that $\mathrm{Re}\:
  A_{\xi} ( R)$ has at most finite codimension in $\mathrm{Re}\: C_{\xi}(\partial R)$.
\end{proof}

We conclude this section by setting notation. All Hilbert spaces are assumed to be complex and separable. Given an $n$-tuple $\mathbb{A}=(A_1,\dots,A_n)$ of operators on a common Hilbert space and $\beta\in\NN^n_0$, we set $\mathbb{A}^\beta=A_1^{\beta_1}\cdots A_n^{\beta_n}$. The general linear group of a Hilbert space $\mathfrak{X}$ is denoted by $\mathrm{GL} ( \mathfrak{X})$, and the unitary group of $\mathfrak{X}$ is denoted by $\mathrm{U} (\mathfrak{X})$.

Let $R$ be a finite Riemann surface and fix a point $x_0\in R$. We denote by $H^{\infty} ( R)$ the Hardy space of bounded analytic functions on $R$. For each analytic function $f:R\to \frk{X}$, denote by $u_f$ the least harmonic majorant of $x \mapsto \| f ( x)\|^2$. We denote by $H^2 ( R, \mathfrak{X})$ the space of all $\mathfrak{X}$-valued analytic functions $f$ for which $u_f ( x_0) < \infty$ and equip the $H^2(R,\frk{X})$ with the Hilbert space norm $\| f \| = \sqrt{u_f ( x_0)}$. Though the norm depends on the choice of $x_0$, the associated topology does not. Let $\grw$ denote harmonic measure at $x_0$. Then $L^2(\pd R,\frk{X})$ is the $\frk{X}$-valued $L^2$-space with norm $g\mapsto \left(\int_{\pd R}\|g(x)\|^2 d\grw(x)\right)^{1/2}$. We observe that every element of $H^2 ( R, \mathfrak{X})$ has $L^2$-boundary values defined $\omega$-a.e. on $\partial R$. We identify $H^2 (R, \mathfrak{X})$ with a subspace of $L^2 ( \partial R, \mathfrak{X})$ via the boundary value map. For details about Hardy spaces of Riemann surfaces, we refer the reader to \cite{Rudin}.

Given $f \in L^{\infty} ( \partial R)$, we denote by $M_f$ the operator of multiplication by $f$ on $L^2(\pd R,\frk{X})$. More generally, let $\frk{Y}$ be another Hilbert space and let $\Theta$ be a bounded (weakly) measurable $\mathcal{B} ( \mathfrak{X},\mathfrak{Y})$-valued function. We denote by $M_{\Theta}$ the map $L^2(\pd R,\frk{X})\ni f\mapsto \Theta f$.

%%%%%%%%%%%%%%%%%%%%%%%%%%%%%%%%%%%%%%%%%%%%%%%%%%%%%%%%%%%%%%%%%%%%%%%%%%%%%%%%%
\section{Remarks on $H^\infty$-invariant subspaces}

This section is devoted to adapting some results of {\cite{AnD}} to the setting of
finite Riemann surfaces. Most of the results we present here are
straightforward modifications of analogous results in {\cite{AnD}} that can
be demonstrated with essentially the same proofs. In such cases, we simply
state the result and refer the reader to the appropriate point in {\cite{AnD}}.

For this section, we fix a finite Riemann surface $R$ and a point $x_0\in R$. By
{\cite[Sec. IV.5]{FnK}} there is an analytic covering map $\tau : \mathbb{D}
\rightarrow R$ such that $\tau ( 0) = x_0$ and a group $G$ consisting of all
analytic self-maps $\gamma$ of $\mathbb{D}$ for which $\tau \circ \gamma =
\tau$. We call $G$ the \emph{deck transformation group} of $R$ and recall that $R$ is homeomorphic to $\DD/G$ when $G$ is endowed with the discrete topology. Given a collection of functions $\mathcal{F}$ on $\mathbb{D}$, we denote by $\mathcal{F}^G$ the set of all $f \in
\mathcal{F}$ such that $f \circ \gamma = f$ for each $\gamma \in G$. For
example, $H^{\infty} ( \mathbb{D})^G$ denotes the set of all $G$-invariant
bounded analytic functions on $\mathbb{D}$. The main result of this section
is a Beurling-type theorem for the `pure' $H^{\infty} (
\mathbb{D})^G$-invariant subspace of $L^2 ( \partial \mathbb{D},
\mathfrak{X})^G$, where $\mathfrak{X}$ is a Hilbert space.

%\paragraph{Vector bundles.}
We begin with a short discussion of analytic vector bundles over $R$. A
\emph{family of Hilbert spaces} over $R$ is a topological space $E$
together with a continuous projection $p : E \rightarrow R$ such that for each
$x \in R$ the fiber $E_x = p^{- 1} ( \{ x \})$ is a Hilbert space in the
topology inherited from $E$. Given a Hilbert space $\mathfrak{X}$, a
\emph{coordinate covering} for $E$ is a collection of pairs $\{ (
\varphi_i, U_i) \}_{i \in I}$ such that
\begin{enumerate}
  \item[\rm{(i)}] $\{ U_i \}_{i \in I}$ is an open covering of $R$;
  
  \item[\rm{(ii)}] $\varphi_i$ is a homeomorphism of $U \times \mathfrak{X}$ onto $p^{-
  1} ( U_i)$ for each $i$; and
  
  \item[\rm{(iii)}] for each $x \in R$ and $U_i \ni x$, the map $\varphi^x : \mathfrak{X}
  \rightarrow E_x$ given by $\varphi^x h = \varphi ( x, h)$ is a continuous
  linear isomorphism.
\end{enumerate}
A family of Hilbert spaces $E$ over $R$ with a coordinate covering
$\{ ( \varphi_i, U_i) \}_{i \in I}$ is called a \emph{vector bundle}. In this case $x\mapsto\dim E_x$ is constant on $R$; we call this constant the \emph{rank} of the bundle. A \emph{transition map} of the bundle is any map of the form $U_i \cap U_j \ni x \mapsto ( \varphi_i^x)^{- 1} \varphi_j^x$ for which $U_i \cap U_j \neq
\emptyset$. Observe that the transition maps take values in
$\mathrm{\mathrm{GL}} ( \mathfrak{X})$. The \emph{trivial bundle} is the
vector bundle $R \times \mathfrak{X}$ with the obvious projection map and the
single element coordinate covering provided by the identity map of $R \times
\mathfrak{X}$.

%\paragraph{Morphisms.}
Let $E$ and $F$ be vector bundles over $R$ with fiber $\mathfrak{X}$. We say
that a homeomorphism $\Lambda$ from $E$ onto $F$ is a \emph{vector bundle
isomorphism} if for each $x \in R$ the restriction $\Lambda^x = \Lambda |E_x$
is a continuous linear isomorphism of $E_x$ onto $F_x$. We say that a bundle
is \emph{topologically trivial} if it is isomorphic to the trivial bundle.

If each transition map $x \mapsto ( \varphi_i^x)^{- 1} \varphi_j^x$ is an
analytic $\mathrm{GL} ( \mathfrak{X})$-valued function, then $E$ and the
coordinate covering $\{ ( \varphi_i, U_i) \}_i$ are said to be
\emph{analytic}. The trivial bundle over $R$ with fiber $\mathfrak{X}$ is
analytic in the obvious way. If $F$ is an analytic vector bundle over $R$ with
fiber $\mathfrak{X}$ and an analytic coordinate covering $\{ ( \psi_j, V_j)
\}_{j \in J}$, then a vector bundle isomorphism $\Lambda : E \rightarrow F$ is
\emph{analytic} if $\{ ( \psi_j, V_j) \}_{j \in J} \cup \{ ( \Lambda \circ
\varphi_i, U_i) \}_{i \in I}$ also provides an analytic coordinate covering
for $F$. We say that an analytic vector bundle is \emph{analytically trivial} is it analytically isomorphic to the
trivial bundle.

\begin{theorem}[{\cite[Thm. 8.2]{Bungart}}]\label{Bungart}
  Every analytic vector bundle over a non-compact Riemann surface is analytically trivial.
\end{theorem}

A coordinate covering for a vector bundle $E$ with fiber $\mathfrak{X}$ is
called a \emph{flat unitary} coordinate covering if the image of each
transition map is contained in $\mathrm{U} ( \mathfrak{X})$, in which case
$E$ is called a \emph{flat unitary vector bundle}. We note that the
transition maps of a flat unitary coordinate covering are always locally constant, and thus every flat unitary vector bundle over $R$ is also analytic. Let $F$ be another flat unitary
vector bundle over $R$ with fiber $\mathfrak{X}$, and let $\{ ( \varphi_i,
U_i) \}_{i \in I}$ and $\{ ( \psi_j, V_j) \}_{j \in J}$ be coordinate
coverings for $E$ and $F$, respectively. A vector bundle isomorphism
$\Lambda : E \rightarrow F$ is a \emph{flat unitary vector bundle
isomorphism} if $\{ ( \psi_j, V_j) \}_{j \in J} \cup \{ ( \Lambda \circ
\varphi_i, U_i) \}_{i \in I}$ is also a flat unitary coordinate covering for $F$.

In contrast to analytic equivalence, there exist flat unitary vector bundles which are not equivalent as flat unitary vector bundles. To be more precise, let $\pi_1 (R)$ denote the fundamental group of $R$, and write $\alpha_1 \sim \alpha_2$ for $\alpha_1, \alpha_2 \in \mathrm{Hom} ( \pi_1 ( R), \mathrm{U} (\mathfrak{X}))$ whenever there is a $U \in \mathrm{U} ( \mathfrak{X})$ such that $\alpha_1 = U \alpha_2 ( \cdot) U^{- 1}$. One readily verifies that $\sim$ determines an equivalence relation on $\mathrm{Hom} ( \pi_1 ( R),\mathrm{U} ( \mathfrak{X}))$. For the proof of the following theorem, we refer the reader to \cite[Lemma 27]{Gunning}.

\begin{theorem}\label{AnDThmB}
  There is a bijection between $\mathrm{Hom} ( \pi_1 ( R), \mathrm{U} ( \mathfrak{X})) / \sim$ and the set of equivalence classes of flat unitary vector bundles over $R$ with fiber $\mathfrak{X}$.
\end{theorem}

%\paragraph{Extension to $R'$.}
Recall that a finite Riemann surface $R$ is a non-compact subdomain of a
compact Riemann surface for which $\partial R$ is locally an analytic curve.
In particular, $\partial R$ consists of a finite disjoint union of simple
closed analytic curves. From this we easily find another finite Riemann
surface $R' \supset \overline{R}$ such that $R$ is a deformation retract of
$R'$, implying in particular that $R'$ and $R$ have isomorphic fundamental
groups. Given a vector bundle $E'$ over $R'$ with fiber $\mathfrak{X}$ and
projection $p'$, the space $E' |R = \{ f \in E' : p' ( f) \in R \}$ is a
vector bundle over $R$ with fiber $\mathfrak{X}$ and projection $p' | ( E'
|R)$. As observed in \cite[Sec. 1.4]{AnD}, the following is a corollary to the proof of Theorem
3.2.

\begin{corollary}\label{AndExtend}
  If $E$ is a flat unitary vector bundle over $R$, then there is a flat
  unitary vector bundle $E'$ over $R'$ such that $E' |R$ and $E$ are
  equivalent as flat unitary vector bundles.
\end{corollary}

Let $E$ be a flat unitary vector bundle over $R$ with fiber $\mathfrak{X}$ and
a coordinate covering $\{(\varphi_i, U_i) \}_{i \in I}$, and fix a point
$x_0\in R$. Given $x\in U_i\cap U_j$ and $v\in E_x$, we note that $ \|(\varphi_i^x)^{-1}x \|=\|(\varphi_j^x)^{-1}v\|$. An \emph{analytic section} of $E$ is a
continuous map $f : R \rightarrow E$ such that $p \circ f$ is the identity map
on $R$ and $x \mapsto ( \varphi_i^x)^{- 1} \circ f$ is analytic on $U_i$ for
each $i \in I$. We denote by $\Gamma_a ( E)$ the linear space of analytic
sections of $E$. Given $f\in \grG_a(E)$, and define $h_f:R\to\RR$ by setting $h_f(x)=\|(\phi_i^x)^{-1}f(x)\|^2$ when $x\in U_i$. We note that $h_f(x)$ does not depend on which coordinate neighborhood of $x$ we use, and that $h_f$ has a least harmonic majorant $u_f$. The $E$-valued $H^2$ space of $R$ is the Hilbert space $H^2 ( E) = \{ f \in \Gamma_a ( E) : u_f ( x_0) < \infty \}$ with the norm given by $f\mapsto \|f\|=\sqrt{u_f(x_0)}$. We denote by $H^{\infty} ( R)$ the set of all bounded analytic functions on $R$, and note that $H^{\infty} ( R)$ acts on $H^2 ( E)$ by sending $( g, h) \in H^{\infty} ( R) \times H^2 ( E)$ to the analytic section $g h$.

Let $E$ and $F$ be flat unitary vector bundles over $R$ with fiber $\mathfrak{X}$. If $\Lambda : E \rightarrow F$ is a (uniformly) bounded analytic vector bundle isomorphism, then the operator $M_{\Lambda} : f \mapsto\Lambda\of f$ defines a bounded linear isomorphism from $H^2 ( E)$ onto $H^2 ( F)$. One can show, as in \cite[Thm. 1]{AnD}, that if $E$ and $F$ are equivalent flat unitary vector bundles, then there is unitary map $U$ from $H^2 ( E)$ onto $H^2 ( F)$ such that $U ( g h) = g U h$ for all $( g, h) \in H^{\infty} ( R) \times H^2 ( E)$. We now have the following result as corollary of Theorem \ref{Bungart} and Corollary \ref{AndExtend}.

\begin{corollary}
  \label{BundTrivCor}If $E$ is a flat unitary vector bundle over $R$ with fiber $\mathfrak{X}$, then there is a bounded analytic vector bundle isomorphism $\Lambda$ from the trivial bundle $R \times \mathfrak{X}$ onto $E$. In particular, $M_{\Lambda}$ is a bounded linear isomorphism of $H^2 ( R,  \mathfrak{X})$ onto $H^2 ( E)$.
\end{corollary}

%\paragraph{The universal cover and related Hardy spaces.}
The group of deck transformations $G$ for $\tau:\DD\to R$ is a Fuchsian group of the second type that is isomorphic to $\pi_1 (R)$ \cite[\S IV.5]{FnK}. Associated with $G$ is a connected set $D_0 \subset \overline{\mathbb{D}}$, which we choose to contain $0$, with the following properties.
\begin{enumerate}
  \item[\rm{(i)}] The set $\partial \mathbb{D} \cap D_0$ consists of finitely many disjoint arcs in $\pd \DD$, and $\mathbb{D} \cap \partial D_0$ consists of finitely many arcs,
  each of which lies on a circle orthogonal to $\partial \mathbb{D}$.
  
  \item[\rm{(ii)}]  $\{ \gamma ( \mathbb{D} \cap D_0) : \gamma \in G \}$ partitions   $\mathbb{D}$ and $L ( G) = \partial\mathbb{D}\backslash\bigcup_{\gamma\in G} \gamma ( D_0)$ is a set of arc-length measure $0$. 
  
  \item[\rm{(iii)}] The map $\tau$ extends to a local homeomorphism from $D = \bigcup_{\gamma \in G} \gamma ( D_0)$ onto $\overline{R}$. In particular, $\mathbb{D}/ G$ is analytically equivalent to $R$ and $D / G$ is homeomorphic to $\overline{R}$. % Moreover, $\tau|\mathbb{D} \cap D_0$ and $\tau | \partial \mathbb{D} \cap D_0$ are continuous bijections onto $R$ and $\partial R$, respectively.
\end{enumerate}
We refer the reader to \cite[Ch. XI]{Tsuji} for more on this topic. From properties (ii) and (iii) it follows that $\int_{\partial R} u \,d \omega = \int_{\partial \mathbb{D}} u \circ \tau \,d m$ for all $u \in L^1 ( \partial R)$. In this way, the map $f \mapsto f \circ \tau$ determines an isometric isomorphism from $H^p( R)$ and $L^p ( \partial R)$ onto $H^p ( \mathbb{D})^G$ and $L^p ( \partial\mathbb{D})^G$, respectively, for each $p \in [ 1, \infty]$.

Let $\alpha \in \mathrm{Hom} ( G, \mathrm{U} ( \mathfrak{X}))$ and denote by $H^2_{\alpha} ( \mathbb{D}, \mathfrak{X})$ the set of all $f \in H^2 ( \mathbb{D}, \mathfrak{X})$ such that $f \circ \gamma = \alpha ( f) f$ for each $\gamma \in G$. We note that $H^2_{\alpha}( \mathbb{D}, \mathfrak{X})$ is $H^{\infty} ( \mathbb{D})^G$-invariant and that $H_{e}^2 ( \mathbb{D}, \mathfrak{X}) = H^2 ( \mathbb{D},\mathfrak{X})^G$, where $e$ denotes the trivial representation of $G$.

%\paragraph{Connecting back to bundles.}
Theorem \ref{AnDThmB} asserts that $\alpha$ determines an essentially unique flat unitary vector bundle over $R$ with fiber $\mathfrak{X}$. Using a construction of such a bundle, the following theorem is deduced as in \cite[Thm. 5]{AnD}.

\begin{theorem}
  \label{DRIsom}If $\alpha \in \mathrm{Hom} ( G, \mathrm{U} ( \mathfrak{X}))$ and $E$ is the associated flat unitary vector bundle, then $H^{\infty} ( R)$ acting on $H^2_{\alpha}(E)$ is unitarily equivalent to $H^{\infty}(\mathbb{D})^G$ acting on $H^2_{\alpha} ( \mathbb{D}, \mathfrak{X})$.
\end{theorem}

Applying Theorem \ref{DRIsom} to Corollary \ref{BundTrivCor} produces the following.
\begin{corollary}\label{GetPhi}
  There is a bounded $\mathrm{GL} ( \mathfrak{X})$-valued analytic function $\Phi$ on $\mathbb{D}$ with the property that the $M_\Phi|H^2 ( \mathbb{D}, \mathfrak{X})^G$ is a continuous linear isomorphism onto $H_{\alpha}^2 ( \mathbb{D}, \mathfrak{X})$.
\end{corollary}

For each Hilbert space $\mathfrak{X}$ and each $\alpha \in \mathrm{Hom} ( G,\mathrm{U} ( \mathfrak{X}))$ we fix a function $\Phi_{\alpha}$ as given by the
preceding corollary. We remark that if $\Phi'_{\alpha}$ is another such function, then $h \mapsto ( \Phi_{\alpha})^{- 1} \Phi'_{\alpha} h$ is an automorphism of $H^2 ( \mathbb{D}, \mathfrak{X})^G$ and therefore $\Phi_{\alpha} H^2 ( \mathbb{D}, \mathfrak{X})^G = \Phi_{\alpha}' H^2 (\mathbb{D}, \mathfrak{X})^G$.

Let $\mc{A}$ be an algebra of operators acting on a Hilbert space $\frk{H}$, and suppose $\frk{H}'$ is an $\mc{A}$-invariant subspace of $\frk{H}$. The subspace $\frk{H}'$ is said to be \emph{pure} $\mc{A}$-invariant if there is no $\mc{A}|\frk{H}'$-reducing subspace of $\frk{H}'$ on which every element of $\mc{A}|\frk{H}'$ is a normal operator. For example, $\psi H^2(\DD)$ is pure $H^\infty(\DD)$-invariant whenever $\psi$ is an inner function. The following proposition differs from the analogous result in \cite{AnD} due to the fact that $H^\infty(R)$ is not, in general, generated by a single element.

\begin{proposition}\label{PurityProp}
  Let $\frk{M}$ be a $H^\infty(\DD)^G$-invariant subspace of $L^2(\pd\DD,\frk{X})^G$ and let $\frk{M}'$ be the smallest $H^\infty(\DD)$-invariant subspace containing $\frk{M}$.
  \begin{enumerate}
    \item[\rm{(i)}] $\frk{M}'$ is $G$-invariant.
    \item[\rm{(ii)}] $\frk{M}=\frk{M}'\cap L^2(\pd\DD,\frk{X})^G$
    \item[\rm{(iii)}] If $\frk{M}$ is pure $H^\infty(\DD)^G$-invariant, then $\frk{M}'$ is pure $H^\infty(\DD)$-invariant.
  \end{enumerate}
\end{proposition}
\begin{proof}
  The proofs for (i)-(iii) are essentially identical to those found in {\cite[Prop 3.4]{AnD}}, but we present a proof of (iii) to illustrate the role of pure invariance.
  
  We denote by $\grz$ the coordinate function on $\pd\DD$. Suppose $\frk{M}'$ is not a pure $H^\infty(\DD)$-invariant subspace of $L^2(\pd\DD,\frk{X})$, meaning that there is a non-trivial $M_\grz$-invariant subspace $\frk{N}'\sbse \frk{M}'$ such that $M_\grz|\frk{N}'$ is a unitary operator. Given $f\in L^2(\pd\DD,\frk{X})$ and $\grg\in G$, we set $C_\grg f=f\of\grg$. The subspace $\frk{R}'=\bigvee_{\grg\in G}C_\grg\frk{N}'$ of $\frk{M}'$ is obvious $G$-invariant, and
  \[ M_\grz \frk{R}'=\bigvee_{\grg\in G} M_\grz C_\grg\frk{N}'=\bigvee_{\grg\in G} C_\grg M_{\grg^{-1}}\frk{N}'=\frk{R}'. \]
  Thus $M_\zeta|\frk{R}'$ is a unitary operator and so $\frk{R}'$ is $H^\infty(\DD)^G$-reducing.
  
  By \cite[Thm. VI.8]{Helson}, there is a unique weakly measurable projection-valued function $x\mapsto P(x)$ on $\pd \DD$ such that $\frk{R}'=P L^2(\pd\DD,\frk{X})$. Since $\frk{R}'$ is $G$-invariant, we have that $P=P\of\grg$ for each $\grg\in G$, and thus $\frk{R}'$ contains non-zero $G$-invariant elements. By (2) we have that $\frk{K}=\frk{R}'\cap L^2(\pd\DD,\frk{X})^G$ is a non-trivial subspace of $\frk{M}$. One now easily verifies that $\frk{K}$ is $H^\infty(\DD)^G$-reducing with the property that $M_f|\frk{K}$ is normal for any $f\in H^\infty(\DD)^G$.
\end{proof}

We now present the main theorem of this section. The proof of this is essentially contained in the proof of \cite[Thm. 11]{AnD}, but we sketch it here for the reader's convenience.

\begin{theorem}\label{AnDInvarSbspc}
  Suppose $\frk{M}$ a pure $H^\infty(\DD)^G$-invariant subspace of $L^2(\pd\DD,\frk{X})^G$. There exist a Hilbert space $\frk{Y}$, a representation $\gra:G\to\mathrm{U}(\frk{Y})$, and a $\mc{B}(\frk{Y},\frk{X})$-valued weakly measurable function $\Psi$ on $\pd\DD$ with the property that $\Psi(z)$ is an isometry for a.e. $z\in\pd\DD$, that $(\Psi\of\grg)\cdot \gra(\grg)=\Psi$ for each $\grg\in G$, and that
  \[ \frk{M}=\Psi\Phi_\gra H^2(\DD,\frk{Y})^G. \]
\end{theorem}
\begin{proof}
  Let $\frk{M}'$ be the smallest $H^\infty(\DD)$-invariant subspace containing $\frk{M}$. As $\frk{M}$ is pure $H^\infty(\DD)^G$-invariant, it follows from Proposition \ref{PurityProp}(iii) that $\frk{M}'$ is pure $H^\infty(\DD)$-invariant. By \cite[Thm. VI.9]{Helson} there is a Hilbert space $\frk{Y}$ and a weakly measurable $\mc{B}(\frk{X},\frk{Y})$-valued function $\Psi$ such that $\frk{M}'=\Psi H^2(\DD,\frk{Y})$ and $\Psi(z)$ is an isometry for a.e. $z\in\pd\DD$. Given $\grg\in G$ it follows from Proposition \ref{PurityProp}(i) that $(\Psi\of\grg)H^2(\DD,\frk{Y})=\Psi H^2(\DD,\frk{Y})$. A corollary of \cite[Thm. VI.9]{Helson} asserts that there is a $\gra(\grg)\in \mathrm{U}(\frk{Y})$ such that $(\Psi\of\grg)\gra(\grg)=\Psi$. Thus $\gra(\grg)=(\Psi\of\grg)^*\Psi$, from which we readily deduce that $\gra$ is a unitary representation of $G$ on $\frk{Y}$. Proposition \ref{PurityProp}(ii) implies that $\frk{M}=\Psi H^2_\gra(\DD,\frk{Y})$, and the theorem now follows from Corollary \ref{GetPhi}.
\end{proof}

%%%%%%%%%%%%%%%%%%%%%%%%%%%%%%%%%%%%%%%%%%%%%%%%%%%%%%%%%%%%%%%%%%%%%%%%%%%%%%%%%%
\section{Virtual Similarity}

Let $\mathbb{V}$ and $\mathbb{W}$ denote $n$-tuples of commuting isometries on Hilbert spaces $\mathfrak{H}$ and $\mathfrak{K}$, respectively. Recall that $\mathbb{V} \lesssim \mathbb{W}$ if there is a finite codimensional $\mathbb{W}$-invariant subspace $\mathfrak{K}' \subseteq \mathfrak{K}$ such that $\mathbb{W}|\mathfrak{K}'$ is similar to $\mathbb{V}$. If $\WW\lesssim\VV$ as well, then we say $\VV$ is \emph{virtually similar} and write $\VV\approx\WW$. The following is easily deduced from this definition.

\begin{corollary}\label{VS:ApproxEqAnn}
  If $\mathbb{V}$ and $\mathbb{W}$ are virtually similar, then
  $\mathrm{Ann} ( \mathbb{V}) = \mathrm{Ann} ( \mathbb{W})$.
\end{corollary}

\begin{lemma}
  \label{BlashProp}Let $V$ be an isometry on a Hilbert space $\mathfrak{H}$. If $\mathfrak{H}'$ is a finite codimensional $V$-invariant subspace and $V$ has no eigenvalues, then there is a finite Blaschke product $B$ such that $B (V)\mathfrak{H} \sbse \mathfrak{H}' \sbse \mathfrak{H}$. Moreover, if $V$ is a shift of finite multiplicity, then $B(V)\frk{H}$ has finite codimension in $\frk{H}'$.
\end{lemma}
\begin{proof}
 The compression of $V$ to $(\frk{H}')^\bot$ has a minimal polynomial $Q$. Denote by $\grl_1,\dots,\grl_m$ the roots of $Q$, and assume only the first $\ell$ are contained in $\DD$. For $j>\ell$, we note that $(V-\grl_jI)\frk{H}$ is dense in $\frk{H}$. Thus we set $B(z)=\prod_{j=1}^\ell \frac{z-\grl_j}{1-\cc{\grl_j}z}$. 
 
 If $V$ is a shift of finite multiplicity, then the same is true of $\frac{V-\grl_iI}{I-\cc{\grl_i}V}$ for $i=1,\dots,\ell$. Thus $B(V)=\prod_{j=1}^\ell\frac{V-\grl_jI}{I-\cc{\grl_j}V}$ is also a shift of finite multiplicity.
\end{proof}

\begin{corollary}
  Let $\VV$ be an $n$-tuple of commuting shifts of finite multiplicity on a Hilbert space $\frk{H}$. If $\mathfrak{H}'$ is a $\mathbb{V}$-invariant subspace of finite codimension, then  $\mathbb{V}|\mathfrak{H}' \approx \mathbb{V}$.
\end{corollary}

\begin{lemma}\label{LessSUFEq}
 Suppose that $\VV$ and $\WW$ are $n$-tuples of commuting isometries. If $\VV\lesssim \WW$ and $W_n$ has no eigenvalues, then each element of $\VV$ is a shift of finite multiplicity if and only if each element of $\WW$ is a shift of finite multiplicity. In either case $\VV\approx\WW$.
\end{lemma}
\begin{proof}
   Let $\frk{K}'$ be a $\WW$-invariant subspace of finite codimension in $\frk{K}$, and $S:\frk{H}\to \frk{K}'$ a boundedly invertible operator such that $SV_i=W_iS$ for $i=1,\dots,n$. By Lemma \ref{BlashProp}, there is a Blaschke product $B$ such that $B(W_n)\frk{K}\sbse \frk{K}'$, whence
   \[ B(W_n)\bigcap_{j=1}^\infty W_i^j\frk{K}\sbse S\bigcap_{j=1}^\infty V_i^j\frk{H} \sbse \bigcap_{j=1}^\infty W_i^j\frk{K}. \]
   Thus $V_i$ is a shift if and only if $W_i$ is a shift.
   
   Plainly
   \begin{equation}\label{IncluChain1}
      W_iB(W_n)\frk{K}\sbse W_i\frk{K}'\sbse \frk{K}'\sbse \frk{K}.
   \end{equation}
   If $W_i$ and $W_n$ are shifts of finite multiplicity, then $W_iB(W_n)\frk{K}$ has finite codimension in $\frk{K}$, and it follows from \eqref{IncluChain1} that $SV_i\frk{H}$ has finite codimension in $S\frk{H}$. That is, if each element of $\WW$ is a shift of finite multiplicity, then each element of $\VV$ is a shift of finite multiplicity. We also note that
   \begin{equation}\label{IncluChain2}
    W_iB(W_n)\frk{K}'\sbse W_iB(W_n)\frk{K}\sbse B(W_n)\frk{K} \sbse \frk{K}'. 
   \end{equation}
   If $V_i$ and $V_n$ are shifts of finite multiplicity, then $W_iB(W_n)\frk{K}'$ has finite multiplicity in $\frk{K}'$. It then follows from \eqref{IncluChain2} that $W_iB(W_n)\frk{K}$ has finite codimension in $B(W_n)\frk{K}$. Because $f\mapsto B(W_n)f$ is isometric on $\frk{K}$, it follows that each element of $\WW$ is a shift of finite multiplicity if and only if the same holds for each element of $\VV$.
   
   Assume that $\WW$ is an $n$-tuple of shifts of finite multiplicity, and note that $B(W_n)\frk{K}$ has finite codimension in $\frk{K'}$. Thus the $\VV$-invariant subspace $\frk{H}'=S^{-1}B(W_n)\frk{K}$ has finite codimension in $\frk{H}$. For $f\in\frk{K}$,
   \[ B ( W_n) W_j f = W_j S S^{- 1} B ( W_n) f = S V_j S^{- 1} B ( W_n) f, \]
   and so $S^{-1}B(W_n)$ intertwines $\WW$ and $\VV$. Because $g\mapsto S^{-1}B(W_n)g$ is one-to-one from $\frk{K}$ onto $\frk{H}'$, we conclude that $\WW\lesssim\VV$ as well.
\end{proof}

Recall that the virtual cyclicity $\kappa ( \mathbb{V})$ of $\mathbb{V}$ is
the smallest positive integer $k$ for which there exists a set of vectors $h_1,
\ldots, h_k \in \mathfrak{H}$ such that $\bigvee_{j = 1}^k \{p(\VV)h_j:p\in\CC[x_1,\dots,x_n]\}$ has finite codimension. We note that
if $\mathfrak{H}'$ is a finite codimensional $\mathbb{V}$-invariant subspace
of $\mathfrak{H}$, then $\kappa ( \mathbb{V}) = \kappa (
\mathbb{V}|\mathfrak{H}')$. Thus $\mathbb{V}$ is always virtually similar to
an $n$-tuple that is both $\kappa ( \mathbb{V})$-cyclic and virtually $\kappa
( \mathbb{V})$-cyclic. We remark that a virtually cyclic $n$-tuple need not be cyclic.

\begin{example}
  Let $V_1$ and $V_2$ be the isometries on $H^2(\DD,\CC^2)$ given by the equations $V_1(f,g)(z)=(zf(z),zg(z))$ and $V_2(f,g)(z)=(zg(z),zf(z))$. The $(V_1,V_2)$-invariant subspace $\frk{M}$ generated by any $h\in H^2(\DD,\CC^2)$ has codimension at least 1. Indeed, if $h(0)=(a,b)\in\CC^2$ is non-zero, then $(\cc{b},-\cc{a})$ is orthogonal to $\frk{M}$. In the case that $h(0)=0$, then $\frk{M}\sbse \grz H^2(\DD,\CC^2)$, where $\grz$ is the coordinate function on the disc. Thus $(V_1,V_2)$ is not cyclic. 
\end{example}

\begin{lemma}\label{KappaLem}
  Let $\VV$ and $\WW$ be $n$-tuples of commuting shifts of finite multiplicity. If $\VV\lesssim\WW$, then $\grk(\VV)=\grk(\WW)$.
\end{lemma}

\begin{proof}
  By Lemma \ref{LessSUFEq}, we have $\VV\approx \WW$. Suppose $\VV$ and $\WW$ act on Hilbert spaces $\frk{H}$ and $\frk{K}$, respectively. There is a finite codimensional $\WW$-invariant subspace $\frk{K}'\sbse \frk{K}$ such that $\WW|\frk{K}'$ is similar to $\VV$; say $S\in\mc{B}(\frk{H},\frk{K}')$ is boundedly invertible and $SV_i=W_iS$ for $i=1,\dots,n$. Let $k=\grk(\WW)=\grk(\WW|\frk{K}')$ and fix $f_1,\dots,f_k\in\frk{K}'$ that determine a $\WW$-invariant subspace $\frk{K}''$ of finite codimension in $\frk{K}'$. The subspace $S\frk{K}''=\bigvee_{i=1}^k\bigvee_{\grb\in\NN_0^n}\VV^\grb Sf_i$ has finite codimension in $\frk{H}$, and therefore $\grk(\VV)\leq k$. Because $\VV\approx \WW$, a similar argument proves that $\grk(\WW)\leq\grk(\VV)$ as well.
\end{proof}

%%%%%%%%%%%%%%%%%%%%%%%%%%%%%%%%%%%%%%%%%%%%%%%%%%%%%%%%%%%%%%%%%%%%%%%%%%
\subsection{The Special Case where $\mathrm{Ann} ( \mathbb{V})$ is Prime.}\label{Sec:SC}

Throughout this subsection we assume that $\mathrm{Ann}(\mathbb{V})$ is a
prime ideal and set $\mathcal{V}= Z ( \mathrm{Ann} ( \mathbb{V}))$. Let $R$ be the finite Riemann surface and $\xi$ the map from $\overline{R}$ onto $\mc{V}\cap\overline{\DD}^n$ given by Proposition \ref{GetR}. Writing $\xi=(\xi_1,\dots,\xi_n)$, we note that $\xi_1, \ldots,\xi_n$ are unimodular on $\partial R$ and thus determine isometric multiplication operators on $H^2( R, \mathfrak{X})$ for any Hilbert space $\mathfrak{X}$. We abbreviate the $n$-tuple of multiplication operators $(M_{\xi_1},\dots,M_{\xi_n})$ by $\mathbb{M}_\xi$. Fix $x_0 \in R$ and
let $\omega$ denote harmonic measure at $x_0$. Recall that $A_{\xi} ( R)$
is the uniform closure in $A(R)$ of the unital algebra generated by
$\xi_1, \ldots, \xi_n$. We require the following result, whose proof is contained in that of \cite[Lemma 3.4]{AKM} for $n=2$ and \cite[Lemma 7.18]{Timko} for $n>2$. Here we sketch the proof for the reader's convenience.

\begin{lemma}\label{AbsContLem}
  Let $\nu$ be a diffuse finite positive measure on $\pd R$ and denote by $\WW$ the $n$-tuple $(\xi_1,\dots,\xi_n)$ acting by multiplication on the $L^2(\nu)$-closure of $A_\xi(R)$. If $\WW$ an $n$-tuple of commuting shifts of finite multiplicity, then $\nu\ll\grw$. 
\end{lemma}
\begin{proof}
	Denote by $A^2_{\xi}(\nu)$ and $A^2(\nu)$ the $L^2(\nu)$-closures of $A_\xi(R)$ and $A(R)$, respectively, and denote by $\mathbb{U}$ the $n$-tuple $(\xi_1,\dots,\xi_n)$ acting on $A^2(\nu)$ by multiplication. Note that $A_\xi^2(\nu)$ has finite codimension in $A^2(\nu)$ and $\WW=\mathbb{U}|A_\xi^2(\nu)$, Because $\nu$ has no atoms, $U_n$ has no eigenvalues. By Lemma \ref{LessSUFEq}, it follows that $U_i$ is a shift of finite multiplicity for $i=1,\dots,n$.

We decompose $\nu$ as $hd\grw+d\nu_s$, where $\nu_s\bot \grw$ and $h$ is a non-negative element of $L^1(\grw)$. Because $A(R)$ is a hypo-Dirichlet algebra \cite[Lem. 1]{Wermer}, it follows from \cite[Sec. 3]{AhrnSar} that every representing measure for the character $f\mapsto f(x_0)$ on $A(R)$ is absolutely continuous with respect to $\grw$. Here we use the fact that $\grw$ is an Arens-Singer measure. By \cite[Lem. II.7.4]{Gamelin}, there exists an $F_\grs$-set $E$ of harmonic measure $0$ such that $\nu_s(\pd R\bksl E)=0$. Applying Forelli's Lemma, we find that $\chi_E\in A^2(\nu)$, which is to say that $A^2(\nu)=A^2(hd\grw)\oplus A^2(\nu_s)$. Note that $(\nu_s\of\xi_n^{-1})(\pd\DD\bksl\xi_n(E))=0$. As $\xi_n$ is piecewise smooth on $\pd R$ and $\grw(E)=0$, it follows that $\xi_n(E)\sbse \pd\DD$ has arc-length measure $0$. That is, $\nu_s\of\xi_n^{-1}$ is singular relative to Lebesgue measure on $\pd \DD$. By the Kolmogorov-Krein theorem for the disc
\[ 0=\inf_{f\in A(\DD)}\int_{\pd\DD}|1-z f(z)|^2d(\nu_s\of\xi_n^{-1})(z)=\inf_{f\in A(\DD)}\int_{\pd R}|1-\xi_n(f\of\xi_n)|^2d\nu_s. \]
That is, $\chi_E\in \xi_n\mathrm{clos}_{L^2(\nu)}(\CC[\xi_n]\chi_E)$ and $f\mapsto \xi_n f$ is a unitary operator on $\mathrm{clos}_{L^2(\nu)}(\CC[\xi_n]\chi_E)$. Because $f\mapsto \xi_n f$ on $A^2(\nu)$ is a shift, it follows that $\nu_s(E)=\|\chi_E\|^2_{L^2(\nu)}=0$.
\end{proof}

In the following, we denote by $A(R,\CC^k)$ the linear space of all continuous functions $f:\cc{R}\to\CC^k$ that are analytic on $R$. If each component of $f$ is an element of $A_\xi(R)$ as well, then we write $f\in A_\xi(R,\CC^k)$. We view both $A(R,\CC^k)$ and $A_\xi(R,\CC^k)$ as subspaces of $L^2(\pd R,\CC^k)$.

\begin{lemma}\label{SubNormLem}
  Assume that $\VV$ has cyclic set of size $k$.
  \begin{enumerate}
    \item[\rm{(i)}] There exists a $k \times k$ matrix-valued measurable function $\Gamma$ on $\partial R$ such that $\VV$ is unitarily equivalent to $\mathbb{M}_{\xi} |\mathfrak{N}$, where
      \[ \mathfrak{N}=\mathrm{clos}_{L^2(\pd R,\CC^k)}\{ \Gamma f : f \in A_{\xi} (R,\CC^k) \}. \]
    \item[\rm{(ii)}] There exists a pure $H^{\infty} (R)$-invariant subspace of finite codimension in $\mathfrak{N}$.
  \end{enumerate}
\end{lemma}

\begin{proof}
  Let $\widetilde{\VV}$ denote the minimal unitary extension of $\VV$ to a Hilbert space $\widetilde{\frk{H}}$, and let $h_1, \ldots, h_k \in \mathfrak{H}$ form a cyclic set for $\mathbb{V}$. There is a projection valued measure $E$ concentrated on $\mathcal{V} \cap ( \partial \mathbb{D})^n$, coming from $\widetilde{\VV}$, such that
  \[ \langle p ( \mathbb{V}) h_i, q ( \mathbb{V}) h_j \rangle =
     \int_{\mathcal{V} \cap ( \partial \mathbb{D})^n} p ( z) \overline{q (
     z)} d \mu_{i j} ( z), \hspace{2em} d \mu_{i j} = \langle d E \cdot h_i,
     h_j \rangle, \]
  for $i,j=1,\dots,k$. From this we readily deduce that each $\mu_{i j}$ is absolutely continuous
  with respect to $\mu=\sum_i \mu_{i i}$.
  
  We claim that $\mu_{i i}$ has no atoms. Take $w \in \mathbb{C}^n$ and note that $\mu_{i i} ( \{ w \}) = \| E ( \{ w \}) h_i \|^2$. Fixing $v \in E ( \{ w \})$, we we find that for any $\beta \in (\NN_0)^{n}$, $\ell \in \{ 1, 2, \dots \}$, and $g \in  \mathfrak{H}$,
  \[ | \langle v, \tilde{\mathbb{V}}^{\ast \beta} g \rangle | = | w^{\beta}
     w_1^{\ell}|\cdot| \langle v, \widetilde{V}_1^{\ell} g \rangle | = |\langle v, V_1^{\ell} g \rangle | = | \langle P_{\mathfrak{H}}
     v, V_1^{\ell} g \rangle | . \]
  Sending $\ell \rightarrow \infty$, we find that $v$ is orthogonal to vectors of the form $\tilde{\mathbb{V}}^{\ast \beta} g$. As the set of such vectors is dense in $\tilde{\mathfrak{H}}$, it follows that $v = 0$ and $\mu_{ii}$ has no atoms.
  
	Because $\xi$ sends a cofinite subset of $\pd R$ homeomorphically onto a cofinite subset of $\mc{V}\cap(\pd\DD)^n$, the pull-back measure $\nu_{i i} = \mu_{i i} \circ \xi$ is well-defined and defuse. The restriction of $\VV$ to $\bigvee_{\beta\in\NN_0^n} \mathbb{V}^{\beta} h_i$ is unitarily equivalent to $\mathbb{M}_\xi$ restricted to $A^2_\xi(\nu_{ii})$ and so, by Lemma \ref{AbsContLem}, we have that $\nu_{i i} \ll \omega$. In particular, $\nu=\sum_i\nu_{ii}$ is absolutely continuous with respect to $\grw$. Let $\grG$ be given by $( \Gamma^2)_{i j} = \left( \frac{d \mu_{j
  i}}{d \mu} \circ \xi \right) \frac{d \nu}{d \omega}$ for each $i$ and $j$, and note that
  \[ \langle p ( \mathbb{V}) h_i, q ( \mathbb{V}) h_j \rangle =
     \int_{\mathcal{V} \cap ( \partial \mathbb{D})^n} \sum_{\ell = 1}^k (
     \Gamma_{\ell i} \cdot p \circ \xi) \overline{( \Gamma_{\ell j} \cdot q
     \circ \xi)} d \omega \]
     for $p,q\in\CC[x_1,\dots,x_n]$. Assertion (i) is now proved.

  Recall that there is a single-variable polynomial $Q$ such that $Q(\theta_n)A(R)$ is contained with finite codimension in $A_{\theta} ( R)$. Therefore $\mathfrak{M} = \mathrm{clos}_{L^2(\grw)} \{ Q(\theta_n)\Gamma f : f \in A(R,\CC^k)  \}$ has finite codimension in $\mathfrak{N}$. Because $\mathrm{clos}_{L^2(\grw)} A (R) = H^2 (R)$ (see \cite[pg. 168]{GamAndLum}), we conclude that $\mathfrak{M}$ is $H^{\infty}(R)$-invariant. It remains to show that $\frk{M}$ is pure invariant. Denote by $U:\frk{N}\to\frk{H}$ the unitarily equivalence given by assertion (i), and let $\frk{R}\sbse\frk{M}$ be a $H^\infty(R)|\frk{M}$-invariant subspace on which each element of $H^\infty(R)$ acts as a normal operator. Then $M_{\xi_1}|\frk{R}$ is a unitary operator, whence $V_1|U\frk{R}$ is unitary. But $V_1$ is a shift and therefore has no unitary summands. That is, $\frk{R}=\{0\}$ and $\frk{M}$ is pure invariant.
\end{proof}

Let $\tau : \mathbb{D} \rightarrow R$ denote the universal covering map appearing in Section 3, and let $G$ denote the group of deck transformations associated with $\tau$. We set $\eta=(\eta_1,\dots,\eta_n)=\xi\of\tau$ and denote by $A_{\eta} (
\mathbb{D})$ the closed unital subalgebra generated by $\eta_1, \ldots, \eta_n$ in the disc algebra $A(\DD)$, and by $A^2_\eta(\DD)$ the $L^2$-closure of $A_\eta(\DD)$. Note that each $\eta_i$ is a
$G$-invariant inner function on $\mathbb{D}$. Given a representation $\alpha
: G \rightarrow \mathrm{U} ( \mathbb{C}^k)$, we fix a $k\times k$ matrix-valued analytic function $\Phi_{\alpha}$ on $\mathbb{D}$ so that $f
\mapsto \Phi_{\alpha} f$ is boundedly invertible from $H^2 ( \mathbb{D},
\mathbb{C}^k)^G$ onto $H^2_{\alpha} ( \mathbb{D}, \mathbb{C}^k)$.

\begin{lemma}
  \label{UniModel} There is a finite codimensional $\mathbb{V}$-invariant subspace $\mathfrak{H}'\sbse\frk{H}$ such
  that $\mathbb{V}|\mathfrak{H}'$ is similar to $\mathbb{M}_{\eta} |A^2_{\eta} (
  \mathbb{D}, \mathbb{C}^k)$ for $k = \kappa ( \mathbb{V})$.
\end{lemma}

\begin{proof}
  Without loss of generality, we assume that $\VV$ is also $k$-cyclic. By Lemma \ref{SubNormLem}(ii), there is a $\mathbb{V}$-invariant subspace $\frk{R}$ of
  finite codimension in $\mathfrak{H}$ so that
  $\mathbb{V}|\frk{R}$ is unitarily equivalent to $\mathbb{M}_{\eta} |\mathfrak{M}$
  where $\mathfrak{M}$ is a pure $H^{\infty} ( \mathbb{D})^G$-invariant subspace
  of $L^2 ( \partial \mathbb{D}, \mathbb{C}^k)^G$. Let $U:\frk{M}\to\frk{R}$ denote the unitary equivalence. Thus, by Theorem
  \ref{AnDInvarSbspc}, there is a Hilbert space $\frk{Y}$, a $\mc{B}(\frk{Y},\CC^k)$-valued weakly measurable function $\Psi$
  on $\partial \mathbb{D}$, and a unitary representation $\alpha$ of $G$ on $\frk{Y}$ such that $\mathfrak{M}= \Psi \Phi_{\alpha} H^2 ( \mathbb{D}, \frk{Y})^G$. Moreover, $\Psi|\pd\DD$ is a.e. isometric and thus $k\geq r=\dim\frk{Y}$.
  
  With $e_1,\dots,e_r$ denoting an orthonormal basis for $\frk{Y}$, we set $f_i=\Psi\Phi_\gra e_i$ for $i=1,\dots, r$. Because $A_\eta^2(\DD)$ has finite codimension in $H^2(\DD)^G$, the subspace $\bigvee_{i=1}^n A_\eta(\DD)f_i$ has finite codimension in $\frk{M}$. Thus $\frk{H}'=\bigvee_{i=1}^r\bigvee_{\grb\in\NN_0^k}\VV^\grb Uf_i$ is a subspace of finite codimension in $\frk{H}$. By definition of $\grk(\VV)$, we have that $k\leq r$ and thus $k=r$.
\end{proof}

\begin{theorem}
  \label{MainThm} Let $\VV$ and $\WW$ be $n$-tuples of commuting shifts of finite multiplicity. If $\mathrm{Ann} (\mathbb{V})$ is a prime ideal, then $\mathbb{V} \approx \mathbb{W}$ if
  and only if $\mathrm{Ann} ( \mathbb{V}) = \mathrm{Ann} ( \mathbb{W})$ and
  $\kappa ( \mathbb{V}) = \kappa ( \mathbb{W})$.
\end{theorem}
\begin{proof}
  If $\VV\approx \WW$, then it follows from Corollary \ref{VS:ApproxEqAnn} that $\Ann(\VV)=\Ann(\WW)$,  and from Lemma \ref{KappaLem} that $\grk(\VV)=\grk(\WW)$.

  Assume that $\Ann(\VV)=\Ann(\WW)$ and $\grk(\VV)=\grk(\WW)$. By Lemma \ref{UniModel}, there is an $n$-tuple of commuting shifts of finite multiplicity $\mathbb{U}$, which depends up to similarity only on $\Ann(\VV)$ and $\kappa(\VV)$, such that $\mathbb{U}\lesssim \VV$. By Proposition \ref{LessSUFEq}, it follows that $\mathbb{U}\approx \VV$. A similar argument shows that $\mathbb{U}\approx\WW$ as well.
\end{proof}

%%%%%%%%%%%%%%%%%%%%%%%%%%%%%%%%%%%%%%%%%%%%%%%%%%%%%%%%%%%%%%%%%%%%%%%
\subsection{The General Case}

  If $\mathrm{Ann} ( \mathbb{V})$ is not prime, then there exists a
  unique finite collection of prime ideals $\mathcal{I}_1, \ldots,
  \mathcal{I}_m$ with $m > 1$ such that $\mathrm{Ann} ( \mathbb{V}) =
  \bigcap_{i = 1}^m \mathcal{I}_i$ and $\mathrm{Ann}(\VV)\subsetneq \bigcap_{i\neq j}\mc{I}_i$ for $j=1,\dots,m$. We call these ideals the \emph{prime factors} of $\Ann(\VV)$. For  $i = 1, \ldots, m$ we set $\widehat{\mathcal{I}}_i = \bigcap_{j \neq i} \mathcal{I}_j$, and define the subspaces
  \[ \mathfrak{H}_i = \mathrm{clos} \{ p ( \mathbb{V}) f : p \in
     \widehat{\mathcal{I}}_i, f \in \mathfrak{H} \}, \quad
     \mathfrak{H}_i^+ = \left\{ f \in \mathfrak{H}: p ( \mathbb{V}) f = 0\text{ for all } p \in \mathcal{I}_i \right\}. \]
		 Note that $\frk{H}_i\sbse \frk{H}_i^+$ for each $i$. Because the ideals of $\CC[x_1,\dots,x_n]$ are finitely generated, the $n$-tuple $\VV|\frk{H}_i$ has a finite cyclic set. By \cite[Thm 6.2]{Timko}, each element of $\VV|\frk{H}_i$ has finite multiplicity.

  \begin{lemma}
    Both $\sum_{i=1}^m \mathfrak{H}_i$ and $\sum_{i=1}^m \mathfrak{H}_i^+$ have finite codimension in $\mathfrak{H}$, and
    \begin{equation}\label{GCLem1Eq1}
      \Ann(\VV|\frk{H}_j^+)=\Ann(\VV|\frk{H}_j)=\mc{I}_j,\quad j=1,\dots,m.
    \end{equation}
  \end{lemma}
  \begin{proof}
    Plainly $\frk{H}_i\sbse \frk{H}_i^+$ and
    \begin{equation}\label{GCLem1Eq2}
      Z(\Ann(\VV|\frk{H}_i))\sbse Z(\Ann(\VV|\frk{H}_i^+))\sbse Z(\mc{I}_i),\quad i=1,\dots,m.
    \end{equation}
    Observe that $\VV|\frk{H}_j$ is an $n$-tuple of shifts, and therefore $Z(\VV|\frk{H}_i)$ has no 0-dimensional components \cite[Lemma 3.1]{Timko}. Because $Z(\mc{I}_i)$ is an irreducible variety of dimension 1, we have that $Z(\Ann(\VV|\frk{H}_j))=\mc{I}_j$, and \eqref{GCLem1Eq1} now follows from \eqref{GCLem1Eq2}.
    
    For the remaining assertions, it suffices to show that $\sum_{j=1}^m\frk{H}_j$ has finite codimension. Recall that $V_n$ has finite multiplicity, and thus $\VV$ has a finite cyclic set $\{h_1,\dots,h_k\}$ in $\frk{H}$. Because $Z(\widehat{\mc{I}}_i\cap\widehat{\mc{I}}_j)$ is a finite set whenever $i\neq j$, the ideal $\sum_{\ell=1}^m\widehat{\mc{I}}_\ell$ has finite codimension in $\CC[x_1,\dots,x_n]$; say
    \[ \CC[x_1,\dots,x_n]=\CC\cdot p_1+\cdots+\CC\cdot p_r + \sum_{\ell=1}^m \widehat{\mc{I}}_\ell. \]
    Because each $\widehat{\mc{I}}_i$ is an ideal, we have that
    \[ \frk{H}=\sum_{\ell=1}^r\sum_{j=1}^k \CC p_\ell(\VV)h_j + \bigvee_{j=1}^m \frk{H}_j. \qedhere \]
  \end{proof}
  
  Given a polynomial $p\in\CC[x_1,\dots,x_n]$ and $z\in(\CC\bksl\{0\})^n$, we define $\iota(p)\in\CC[x_1,\dots,x_n]$  by setting
  \[ \iota(p)(z)=z_1^{d_1}\cdots z_n^{\grd_n}\cc{p(1/\cc{z}_1,\dots,1/\cc{z}_n)} \]
  where $d_1,\dots,d_n$ are the degrees of $x_1,\dots,x_n$ in $p$, respectively. We call $\grd=(d_1,\dots,d_n)$ the multi-degree of $p$. We now have
  \[ \iota(p)(\VV)=p(\VV)^*\VV^\grd. \]
  Thus $p\in\Ann(\VV)$ if and only if $\iota(p)\in\Ann(\VV)$. What follows is based on \cite[Thm. 2.1]{AKM} and \cite[Thm. 3.4]{Timko}.
  
  \begin{lemma}
    The subspaces $\frk{H}_1,\dots,\frk{H}_k$ are pairwise orthogonal, the subspaces $\frk{H}_1^+,\dots,\frk{H}_m^+$ are linearly independent, and $\frk{H}_i$ has finite codimension in $\frk{H}_i^+$ for $i=1,\dots,m$.
  \end{lemma}
  \begin{proof}
  Let $p_i\in\widehat{\mc{I}}_i$ and let $p_j\in\widehat{\mc{I}}_j$ for distinct $i,j$ in $\{1,\dots,m\}$. Because $\Ann(\VV|\frk{H}_j)=\mc{I}_j$, it follows that $\iota(p_j)\in\widehat{\mc{I}}_j$ and thus that $\iota(p_j)\cdot p_i\in\Ann(\VV)$. In other words, for $h,h'\in\frk{H}$ and $\grd$ the multi-degree of $p_i$,
  \[ \inner{p_i(\VV)h}{p_j(\VV)h'}=\inner{\VV^\grd}{p_i(\VV)^*\VV^\grd p_j(\VV)h'}= \inner{\VV^\grd}{(\iota(p_i) p_j)(\VV)h'}=0.\]
  That is, $\frk{H}_i$ and $\frk{H}_j$ are orthogonal.
  
  If $f \in \mathfrak{H}_i^+ \cap \sum_{j \neq i}
  \mathfrak{H}_j^+$, then $p ( \mathbb{V}) f = 0$ for each $p \in
  \mathcal{I}_i + \hat{\mathcal{I}}_i$. However, $Z ( \mathcal{I}_i +
  \hat{\mathcal{I}}_j) = \bigcup_{j \neq i} Z ( \mathcal{I}_i) \cap Z (
  \mathcal{I}_j)$ has dimension 0, and thus $\mathcal{I}_i +
  \widehat{\mathcal{I}}_i$ contains a non-zero single-variable polynomial $p_0 ( x_1)$. Because $p_0 ( V_1)$ is injective, we have that $f = 0$.
  
  Let $p\in \widehat{\mc{I}}_j$ for some $j\neq i$ and denote by $\grd$ the multi-degree of $p$. Because $\Ann(\VV|\frk{H}_\ell)=\mc{I}_\ell$ for each $\ell$, we know that $\iota(p)\in \widehat{\mc{I}}_j$. Thus, for $f\in\frk{H}_i^+$ and $g\in\frk{H}$,
   \[ \inner{f}{p(\VV)g}=\inner{\VV^{*\grd}p(\VV)^*\VV^\grd f}{g}=\inner{\iota(p)(\VV) f}{\VV^{\grd}g}=0. \]
   Thus $\frk{H}_i^+\bot \frk{H}_j$ for $j\neq i$, whence $\frk{H}_i^+\ominus \frk{H}_i\sbse \left(\sum_{\ell=1}^m\frk{H}_\ell\right)^\bot$. Because $\sum_{\ell=1}^m\frk{H}_\ell$ has finite codimension in $\frk{H}$, we have that $\frk{H}_i$ has finite codimension in $\frk{H}_i^+$. 
  \end{proof}

  It follows from the preceding lemma that $\mathbb{V}|\mathfrak{H}_i \approx \mathbb{V}|\mathfrak{H}_i^+$ and $\kappa(\mathbb{V}|\mathfrak{H}_i) = \kappa ( \mathbb{V}|\mathfrak{H}_i^+)$ for $i=1,\dots,m$. We also note that if $\mathfrak{H}'$ is a $\mathbb{V}$-invariant finite
  codimensional subspace of $\mathfrak{H}$, then the subspace $\mathfrak{H}_i' = \mathrm{clos} \{ p ( \mathbb{V}) f : p \in \widehat{\mathcal{I}}_i, f \in \mathfrak{H}' \}$ has finite codimension in $\mathfrak{H}_i$. Because of this, $\mathcal{I}_i = \mathrm{Ann} ( \mathbb{V}|\mathfrak{H}_i')$ and $\kappa (
  \mathbb{V}|\mathfrak{H}_i^+) = \kappa ( \mathbb{V}|\mathfrak{H}_i')$. More generally, we have the following.
  
  \begin{lemma}\label{IdlCodimLem}
    Let $\mc{I}$ be an ideal of $\CC[x_1,\dots,x_n]$. The closure $\frk{H}_\mc{I}'$ of $\{p(\VV)h:h\in\frk{H}',p\in\mc{I}\}$ has finite codimension in the closure $\frk{H}_\mc{I}$ of $\{p(\VV)h:h\in\frk{H},p\in\mc{I}\}$. Moreover, $\grk(\VV|\frk{H}_\mc{I}')=\grk(\VV|\frk{H}_\mc{I})$.
  \end{lemma}
  \begin{proof}
    Ideals of $\CC[x_1,\dots,x_n]$ are finitely generated, and thus there are polynomials $p_1,\dots,p_\ell\in\mc{I}$ such that $\mc{I}=\sum_{j=1}^\ell \CC[x_1,\dots,x_n]\cdot p_j$. Because of this, $\frk{H}_\mc{I}'=\bigvee_{j=1}^\ell p_j(\VV)\frk{H}'$, with a similar expression for $\frk{H}_\mc{I}$. Thus, if $f_1,\dots,f_m$ form a basis for $\frk{H}\ominus\frk{H}'$, then
    \[ \frk{H}_\mc{I}=\frk{H}_\mc{I}'+\sum_{j=1}^\ell\sum_{i=1}^m \CC p_j(\VV)f_i. \qedhere \]
  \end{proof}
  
  Recall from Corollary \ref{VS:ApproxEqAnn} that virtually similar $n$-tuples have the same annihilator. In particular, virtually similar tuples have annihilators with the same prime factors.
  
\begin{theorem}\label{GenThm}
  Let $\VV$ and $\WW$ be $n$-tuples of commuting shifts of finite multiplicity on Hilbert spaces $\frk{H}$ and $\frk{K}$, respectively. We denote by $\mc{I}_1,\dots,\mc{I}_m$ the prime factors of $\Ann(\VV)$, and set
  \[ \frk{H}_i^+=\bigcap_{p\in\mc{I}_i}\ker p(\VV), \quad \frk{K}_i^+=\bigcap_{p\in\mc{I}_i}p(\WW), \quad
  i=1,\dots,m.\]
  The following assertions are equivalent.
  \begin{enumerate}
    \item[\rm{(i)}] $\mathbb{V} \approx \mathbb{W}$
    \item[\rm{(ii)}] $\mathrm{Ann} ( \mathbb{V}) = \mathrm{Ann} ( \mathbb{W})$ and $\kappa (
  \mathbb{V}|\mathfrak{H}_j^+) = \kappa ( \mathbb{V}|\mathfrak{K}_j^+)$ for $j=1,\dots,m$.
  \end{enumerate}
\end{theorem}

\begin{proof}
  Assume (ii). As above, we set $\frk{H}_j=\mathrm{clos}\{p(\VV)h:p\in\widehat{\mc{I}}_j,\: h\in\frk{H}\}$ and $\frk{K}_j=\mathrm{clos}\{p(\WW)g:p\in\widehat{\mc{I}}_j,\; g\in\frk{K}\}$ for $j=1,\dots,m$. Because $\sum_{j=1}^m\frk{H}_j$ and $\sum_{j=1}^m\frk{K}_j$ have finite codimension in $\frk{H}$ and $\frk{K}$, respectively, we deduce from Proposition \ref{LessSUFEq} that 
  \[ \VV \approx \bigoplus_{j=1}^m \VV|\frk{H}_j, \quad \WW\approx \bigoplus_{j=1}^m\WW|\frk{K}_j. \]
  As $\Ann(\VV|\frk{H}_i)=\mc{I}_i=\Ann(\WW|\frk{K}_i)$ and $\grk(\VV|\frk{H}_i)=\grk(\WW|\frk{K}_i)$, it follows from Theorem \ref{MainThm} that $\VV|\frk{H}_i\approx \WW|\frk{K}_i$ for $i=1,\dots,m$. Thus we have that
  \[ \VV \gtrsim \bigoplus_{j=1}^m \VV|\frk{H}_j \gtrsim \bigoplus_{j=1}^m\WW|\frk{K}_j\gtrsim \WW. \]
  Likewise, we find that $\WW\gtrsim\VV$.
  
  Conversely, we assume (i) and note that $\Ann(\VV)=\Ann(\WW)$ by Corollary \ref{VS:ApproxEqAnn}. Let $\mathfrak{K}'$ be a finite codimensional
  $\mathbb{W}$-invariant subspace of $\mathfrak{K}$ and let $S\in\mc{B}(\mathfrak{H},\mathfrak{K}')$ be a boundedly invertible operator such that
  $SV_j=W_jS$ for $j=1,\dots,n$. It follows from $\dim(\frk{K}\ominus\frk{K}')<\infty$ and Lemma \ref{LessSUFEq} that each element of $\WW|\frk{K}'$ is a shift of finite multiplicity. We set $\mathfrak{K}_i' = \mathrm{clos} \{ p ( \mathbb{W}) f : p \in \widehat{\mathcal{I}}_i, f \in \mathfrak{K}' \}$ for $i=1,\dots,m$ and note that $\frk{K}_i'$ has finite codimension in $\frk{K}_i$ by Lemma \ref{IdlCodimLem}. Thus $\frk{K}_i'$ has finite codimension in $\frk{K}_i^+$ as well. Because $S^{- 1} \mathfrak{K}_i' =\mathfrak{H}_i$, we conclude that $\kappa ( \mathbb{W}|\mathfrak{K}_i^+) = \kappa (
     \mathbb{W}|\mathfrak{K}_i') = \kappa ( \mathbb{V}|\mathfrak{H}_i) =
     \kappa ( \mathbb{V}|\mathfrak{H}_i^+)$. 
\end{proof}

Before ending this section, we remark that assertion (i) of Theorem \ref{GenThm} can be weakened as follows. Let $\VV$ and $\WW$ be $n$-tuples of commuting shifts of finite multiplicity on Hilbert spaces  $\frk{H}$ and $\frk{K}$, respectively. Suppose there are injective linear maps $X\in\mc{B}(\frk{H},\frk{K})$ and $Y\in\mc{B}(\frk{K},\frk{H})$ such that $\dim (\ran X)^\bot<\infty$, $\dim(\ran Y)^\bot<\infty$, and
\[ XV_i=W_iX, \quad YW_i=V_iX \quad\text{for }i=1,\dots,n. \]
Then we say that $\VV$ and $\WW$ are \emph{virtually quasi-similar}. It is evident that virtual similarity implies virtual quasi-similarity. The converse is also true.

\begin{proposition}
  If $\VV$ and $\WW$ are virtually quasi-similar $n$-tuples of commuting shifts of finite multiplicity, then $\VV\approx \WW$.
\end{proposition}
\begin{proof}
  Let $X$ and $Y$ be as above. If $p\in\Ann(\WW)$, then $0=p(\WW)X=Xp(\VV)$. As $X$ is injective, it follows that $p\in\Ann(\VV)$. In a similar fashion, we see that $\Ann(\VV)\sbse \Ann(\WW)$ as well.

  Let $\mc{I}_1,\dots,\mc{I}_m$ be the prime factors of $\Ann(\VV)$, and set
  \[ \frk{H}_i=\bigvee_{p\in\widehat{\mc{I}}_i}p(\VV)\frk{H}, \quad \frk{K}_i=\bigvee_{p\in\widehat{\mc{I}}_i}p(\WW)\frk{K} \]
  for $i=1,\dots,n$. Plainly $X\frk{H}_i\sbse \frk{K}_i$ and $Y\frk{K}_i\sbse \frk{H}_i$.    By Lemma \ref{IdlCodimLem}, it follows that $\dim(\frk{K}_i\ominus X\frk{H}_i)<\infty$ and $\dim(\frk{H}_i\ominus Y\frk{K}_i)<\infty$.   
  
  With $k=\grk(\VV|\frk{H}_i)$, let $f_1,\dots,f_k\in\frk{H}_i$ such that $\bigvee_{j=1}^k\bigvee_{\grb\in\NN^n_0}\VV^\grb f_j$ has finite codimension in $\frk{H}_i$. This implies that $\bigvee_{j=1}^k\bigvee_{\grb\in\NN^n_0}\WW^\grb Xf_j$ has finite codimension in $\frk{K}_i$, and therefore $\grk(\VV|\frk{H}_i)=k\geq \grk(\WW|\frk{K}_i)$. By a similar argument, we conclude that $\grk(\VV|\frk{H}_i)\leq \grk(\WW|\frk{K}_i)$ as well. The proposition now follows from Theorem \ref{GenThm}.
\end{proof}
%%%%%%%%%%%%%%%%%%%%%%%%%%%%%%%%%%%%%%%%%%%%%%%%%%%%%%%%%%%%%%%%%%%%%%%%%%%%%%%%%%%%
\section{Virtual Unitary Equivalence}

We say that two $n$-tuples of commuting shifts are \emph{virtually unitarily equivalent} if each tuple is unitarily equivalent to a finite codimensional restriction of the other. Fix an $n$-tuple $\VV$ of commuting shifts of finite multiplicity for which $\mathrm{Ann} ( \mathbb{V})$ is prime and set $k=\grk(\VV)$.

\begin{quotation}\noindent
  Question A : If $\mathbb{W}$ is a virtually $k$-cyclic $n$-tuple of commuting shifts of finite multiplicity and $\mathrm{Ann} ( \mathbb{V}) = \mathrm{Ann} ( \mathbb{W})$, is $\mathbb{W}$ virtually unitarily equivalent to $\mathbb{V}$?
\end{quotation}

\noindent For $k=1$, the answer to Question A is `yes', as shown for $n=2$ in \cite{AKM} and $n>2$ in \cite{Timko}. The question remains open in general for $k>1$. We compare Question A with the following related question, where $G$ and $R$ are as in Section \ref{Sec:SC}. 
\begin{quotation}\noindent
  Question B : Given $\alpha \in \Hom ( G, \mathrm{U} ( \mathbb{C}^k))$, does
  there exist a $k\times k$ matrix-valued inner function $\Psi$ on
  $\mathbb{D}$ such that $\Psi H_{\alpha}^2 ( \mathbb{D}, \mathbb{C}^k)$ is
  a finite codimensional subspace of $H^2 ( \mathbb{D}, \mathbb{C}^k)^G$?
\end{quotation}

\noindent By Theorem \ref{AnDInvarSbspc}, every $H^{\infty}(\DD)^G$-invariant subspace of $H^2 ( \mathbb{D}, \mathbb{C}^k)^G$ is of the
form $\Psi H^2_{\alpha} ( \mathbb{D}, \mathbb{C}^k)$ for \emph{some}
representation $\alpha$. If the answer to (B) is `yes', then such a
subspace would exist for \emph{each} representation $\alpha$. 

\begin{proposition}\label{EquivQProp}
	The answer to Question $\mathrm{A}$ is affirmative if and only if the answer to Question $\mathrm{B}$ is affirmative.
\end{proposition}

\begin{proof}
  By Lemma \ref{UniModel}, there is a $\beta\in\Hom(G,\mathrm{U}(\CC^k))$ such that $\VV$ is virtually similar to $\mathbb{M}_\eta|H^2_\grb(\DD,\CC^k)$. Assume that Question A may be answered in the affirmative and let $\gra\in \Hom(G,\mathrm{U}(\CC^k))$. Denote by $\WW$ and $\mathbb{U}$ the $n$-tuples of commuting shifts of finite multiplicity given by $\mathbb{M}_\eta|H^2_\gra(\DD,\CC^k)$ and $\mathbb{M}_\eta|H^2(\DD,\CC^k)^G$, respectively. These $n$-tuples are virtually similar to $\VV$ and thus, by hypothesis, $\VV$ is virtually unitarily equivalent to both $\WW$ and $\mathbb{U}$. In particular, $\WW$ and $\mathbb{U}$ are virtually unitarily equivalent. Thus there exists a finite codimension $\mathbb{U}$-invariant subspace $\frk{M}$ and a unitary operator $T:H^2_\gra(\DD,\CC^k)\to \frk{M}$ such that $U_jT=TW_j$ for $j=1,\dots,n$.
  
  We claim that $T$ is multiplication by a matrix-valued inner function $\Psi$ with the property that $\Psi\of\grg=\Psi\cdot\gra(\grg)^{-1}$ for each $\grg\in G$. Indeed, set $A:f\mapsto T\Phi_\gra f$, and note that $Ae_1,\dots,Ae_k\in H^2(\DD,\CC^k)^G$, where $e_1,\dots,e_k$ form an orthonormal basis for $\CC^k$. Let $Q$ be the polynomial given by Lemma \ref{GetQLem}, and note that multiplication by $Q(\eta_n)g$ commutes with $A$ for each $g\in H^\infty(\DD)^G$. Thus
  \[ Q(\eta_n)g Ae_\ell=A Q(\eta_n)ge_\ell=Q(\eta_n)Age_\ell, \quad \ell=1,\dots,k. \]
  Because $Q(U_n)$ is injective, we have that $A$ commutes with $H^\infty(\DD)^G$ and thus is a multiplication operator with a $G$-invariant analytic matrix-valued symbol $\grJ$. That is, $T$ is multiplication by $\Psi=\grJ\cdot (\Phi_\gra)^{-1}$. The claim now follows from the fact that $T$ is isometric. Because $\frk{M}=\ran T$ has finite codimension, Question B is answered in the affirmative.
  
  It is easily seen that Lemma \ref{LessSUFEq} remains true when `similarity' is replaced by
  `unitary equivalence'. To answer Question A in the affirmative, it therefore suffices to finite a matrix-valued inner function $\Psi'$ with the property that $\Psi' H^2_\grb(\DD,\CC^k)$ has finite codimension in $H^2(\DD,\CC^k)$. That is, an affirmative answer to Question B provides an affirmative answer to Question A.
\end{proof}

While the answer to either Question A or B is not known in general, we can provide an affirmative answer to Question B if we restrict to representations of $G$ with commutative image. In particular, if $G$ is a commutative group, then Question A has an affirmative answer for any $k$. To demonstrate this, we require some additional results, which are based on \cite[Sec. 4]{AKM} and \cite{Khavinson}.

Fix representative curves $K_1,\dots,K_L$ for the generators of the fundamental group of $R$ with base point $x_0$. By the Hurewicz theorem, these curves also provide a basis for the first singular homology group of $R$ with integer coefficients. Through the isomorphism of $G$ with the fundamental group, there are generators $\grg_1,\dots,\grg_L\in G$ such that the curve $K_j$ is covered by a curve $\DD$ beginning at $0$ and ending at $\grg_j(0)$ for $j=1,\dots,L$.

Denote by $\grw_x$ the harmonic measure for evaluation at $x\in R$, and recall that $\grw=\grw_{x_0}$. By Harnack's inequality, the measures $\grw_x$ and $\grw_{x_0}$ are mutually absolutely continuous. We set
\[ P(x,y)=\frac{d\grw_x(y)}{d\grw_{x_0}(y)}, \quad y\in \pd R, \]
and record a few theorems from \cite{Khavinson} that we require.

\begin{lemma}[{\cite[Lemma 3.1]{Khavinson}} ]\label{Khav1}
	For $j=1,\dots,L$, there exists a continuous function $Y_j$ on $\pd R$ with the following property. If $\mu$ is a finite real measure on $\pd R$ and $u(x)=\int_{\pd R}P(x,y)d\mu(x)$, then the period of the harmonic conjugate $*u$ along the curve $K_j$ is given by $\int_{\pd R}Y_j(y)d\mu(y)$.
\end{lemma}

Let $\mc{D}=\{\grD_1,\dots,\grD_L\}$ be a collection of disjoint Borel subsets of $\pd R$, and let $d\nu$ be a finite real Borel measure on $\pd R$. Define the matrix $A(\nu,\mc{D})$ with $ij$-th entry given by
\[ A_{ij}(\nu,\mc{D})=\int_{\grD_j} Y_i(x)d\nu(x), \]
with $i,j\in\{1,\dots,L\}$. This is the period along $K_i$ of the harmonic conjugate of
\[ g_j^{\nu,\mc{D}}(x)=\int_{\grD_j}P(x,y)d\nu(y). \]
We call $A(\nu,\grD)$ the \emph{period matrix} of $\nu$ and $\mc{D}$.

\begin{lemma}[{\cite[Lemma 4.7]{Khavinson}}]\label{Khav2}
	Given a disjoint family of non-empty open arcs $\grD_1,\dots,\grD_L$ in $\pd R$, there exist non-empty arcs $\widetilde{\grD}_1\sbse \grD_1, \dots, \widetilde{\grD}_L\sbse \grD_L$ such that the period matrix for $\grw$ and $\{\widetilde{\grD}_1,\dots,\widetilde{\grD}_L\}$ is non-singular.
\end{lemma}

\begin{lemma}[{\cite[Lemma 4.12]{Khavinson}} ]\label{Khav3}
	Let $\mc{D}=\{\grD_1,\dots,\grD_L\}$ be a disjoint family of non-empty arcs for which the period matrix for $\grw$ and $\mc{D}$ is non-singular, and let $v=(v_1,\dots,v_L)\in\RR^L$. There exists a $v'\in\RR^L$ such that the function $f=\sum_{j=1}^L v'_jg_j^{\grw,\mc{D}}$ has a harmonic conjugate with periods $v_1,\dots,v_L$ along the curves $K_1,\dots,K_L$, respectively.
\end{lemma}

From here on, we assume fix a disjoint family $\mc{D}=\{\grD_1,\dots,\grD_L\}$ of arcs in $\pd R$ such that the period matrix for $\grw$ and $\mc{D}$ is non-singular. For brevity, we set $g_j=g_j^{\grw,\mc{D}}$ for $j=1,\dots,L$. With $f=\sum_{j=1}^L a_j g_j$ for some $a_1,\dots,a_L\in\RR$, we note that $f\of\tau$ has a unique single valued harmonic conjugate $*(f\of\tau)$ on $\DD$ for which $*(f\of\tau)(0)=0$. If $*f$ has periods $b_1,\dots,b_L$, then $*(f\of\tau)\of\grg_j=*(f\of\tau)+b_j$ for $j=1,\dots,L$.

Given a subset $U$ of $\pd R$, we denote by $\chi_U$ the characteristic function of $U$.

\begin{lemma}[{\cite[Lem. 3.11]{AKM}}]\label{AKMFLem}
	Given $b_1,\dots,b_L\in\RR$, there exists a bounded holomorphic function $F$ on  $R$ with finitely many zeros in $R$ satisfying
	\[ \log | F(x) |=-\sum_{i=1}^L b_i\chi_{\grD_i}(x), \quad x\in\pd R. \]
\end{lemma}

Before we construct the matrix-valued inner function appearing in Question B for the case wherein $\gra$ has commutative image, we first construct an analogous scalar-valued inner function.

\begin{lemma}\label{VUE:CharLem}
	Let $\gra$ be a character of $G$. There exists a scalar inner function $\psi$ on $\DD$ such that $\psi H^2(\DD)^G$ has finite codimension in $H^2_\gra(\DD)$.
\end{lemma}
\begin{proof}
	Let $a=(a_1,\dots,a_L)\in \RR^L$ be such that $\gra(\grg_\ell)=e^{ia_\ell}$ for $\ell=1,\dots,L$. By Lemma \ref{Khav3}, there is a $b=(b_1,\dots,b_L)\in\RR^L$ such that $\sum_{j=1}^L b_j (*g_j)$ has period $a_i$ along $K_i$ for $i=1,\dots,L$. The multivalued function $h=\sum_{j=1}^Lb_j(g_j+i *g_j)$ is holomorphic on $R$ with the property that $(\mathrm{Re}\: h)(x)=\sum_{\ell=1}^L b_\ell \chi_{\grD_\ell}(x)$ for $x\in \pd R$. Let $\widehat{h}:\DD\to\CC$ be analytic such that $h\of\tau=\widehat{h}$. With $\grW_\ell=\tau^{-1}(\grD_\ell)$ for $\ell=1,\dots,L$, we note that $(\mathrm{Re}\: \widehat{h})(z)=\sum_{i=1}^L b_i\chi_{\grW_i}(z)$ for a.e. $z\in\pd \DD$. The function $\phi=\exp(\widehat{h})$ and its reciprocal function are both bounded and analytic, and
	\[ \phi\of\grg_\ell = e^{ia_\ell}\phi, \quad \ell=1,\dots,L. \]
In particular, $H^2_\gra(\DD)=\phi H^2(\DD)^G$.

Let $F$ be as in Lemma \ref{AKMFLem}, and set $\psi=(F\of\tau)\phi$. We note that $\log|\psi(z)|=0$ for a.e. $z\in\pd\DD$ and that $\psi\of\grg_\ell=e^{ia_\ell}\psi$ for $\ell=1,\dots,L$. It remains to show that $\psi H^2(\DD)^G$ has finite codimension in $H^2(\DD)^G$, and for this we note the following. If $f\in H^2(R)$ has the same zeros at $F$, counting multiplicity, then $f/F\in H^2(R)$. Thus $(F\of\tau)H^2(\DD)^G$ has finite codimension in $H^2(\DD)^G$, whence $\psi H^2(\DD)^G$ has finite codimension in $H^2_\gra(\DD)$. 
\end{proof}

\begin{corollary}
	If $\alpha$ is a representation of $G$ with commutative image, then the inner function $\Psi$ appearing in Question $\mathrm{B}$ exists. In particular, Question $\mathrm{A}$ is answered in the affirmative when $G$ is commutative.
\end{corollary}

\begin{proof}
	Given complex numbers $c_1,\dots,c_k$, we denote by $\mathrm{diag}(c_1,\dots,c_k)$ the diagonal matrix with entries $c_1,\dots,c_k$. The matrices $\gra(\grg_1),\dots,\gra(\grg_L)$ commute, and thus we assume that there are $a_j^{(1)},\dots,a_j^{(k)}\in\RR$ such that $\gra(\grg_j)=\exp(i\:\mathrm{diag}(a_j^{(1)},\dots,a_j^{(k)}))$ for $j=1,\dots,L$. Denote by $\psi_\ell$ the scalar inner function given by Lemma \ref{VUE:CharLem} for the $L$-tuple $(a_1^{(\ell)},\dots,a_L^{(\ell)})$ for $\ell=1,\dots,k$ and set $\Psi=\mathrm{diag}(\psi_1,\dots,\psi_k)$.
\end{proof}

\end{document}